\documentclass[12pt,twoside]{amsart}
\usepackage{amsfonts}
\usepackage{amsmath, amsthm, amssymb, mathrsfs, paralist,tabularx,supertabular, verbatim, amsthm, manfnt}
\usepackage{url}
\usepackage{comment}
\usepackage[all]{xy}

\newcommand{\Z}{{\mathbb Z}}

\newcommand{\F}{{\mathbb F}}

\include{python}

\evensidemargin \oddsidemargin
\parskip=\medskipamount

\newtheorem{thm}{Theorem}[section]
\newtheorem{theorem}[thm]{Theorem}
\newtheorem{corollary}[thm]{Corollary}
\newtheorem{lemma}[thm]{Lemma}

\newtheorem{conjecture}[thm]{Conjecture}

\theoremstyle{definition}

\theoremstyle{remark}

\newtheorem*{lemma*}{Lemma}

\numberwithin{equation}{section}

\title{On the divisors of $x^n-1$ in $\F_p[x]$}

\author{Lola Thompson}
\address{Department of Mathematics\\ 
6188 Kemeny Hall\\
Dartmouth College\\
Hanover, NH 03755, USA}
\email[] {Lola.Thompson@Dartmouth.edu}

\begin{document}

\begin{abstract} 
\noindent In a recent paper, we considered integers $n$ for which the polynomial $x^n - 1$ has a divisor in $\Z[x]$ of every degree up to $n$, and we gave upper and lower bounds for their distribution. In this paper, we consider those $n$ for which the polynomial $x^n-1$ has a divisor in $\F_p[x]$ of every degree up to $n$, where $p$ is a rational prime. Assuming the validity of the Generalized Riemann Hypothesis, we show that such integers $n$ have asymptotic density $0$. \end{abstract}

\maketitle 

\section{Introduction and statement of results}

In a recent paper \cite{thompson}, we examined the question ``How often does $x^n-1$ have a divisor in $\Z[x]$ of every degree between $1$ and $n$?'' We called an integer $n$ with this property \textit{$\varphi$-practical} and showed that $$\#\{n \leq X: n \ \hbox{is} \ \varphi\hbox{-practical}\} \asymp \frac{X}{\log X}.$$ We examined variants of this question over other polynomial rings in \cite{pollackthompson} and \cite{thompson3}. In \cite{pollackthompson}, Pollack and I extended the notion of $\varphi$-practical by defining an integer $n$ to be $F$-practical if $x^n-1$ has a divisor of every degree between $1$ and $n$ over a number field $F$. We showed that, for any number field $F$, $$\#\{n \leq X : n \ \hbox{is} \ F\hbox{-practical}\} \asymp_F \frac{X}{\log X}.$$ 

We shifted our focus to fields with positive characteristic in \cite{thompson3}. For each rational prime $p$, we defined an integer $n$ to be \textit{$p$-practical} if $x^n-1$ has a divisor in $\F_[x]$ of every degree between $1$ and $n$. Since every $\varphi$-practical number is $p$-practical for all $p$, our work from \cite{thompson} immediately implies that $\#\{n \leq X: n \ \hbox{is} \ p\hbox{-practical}\}$ is at least a positive constant times $\frac{X}{\log X}$. Moreover, we showed in \cite{thompson3} that $$\#\{n \leq X : n \ \hbox{is} \ p\hbox{-practical for all} \ p\} \ll \frac{X}{\log X}$$ and that, for any fixed $p$, $$\#\{n \leq X: n \ \hbox{is} \ p\hbox{-practical but not} \ \varphi\hbox{-practical}\} \gg \frac{X}{\log X}.$$ The difficulty lies in finding an upper bound for the count of integers up to $X$ that are $p$-practical for an arbitrary but fixed prime $p$. This will be the subject of our present investigation.   

For each fixed prime $p$, we define $$F_p(X): = \# \{n \leq X : n \ \hbox{is} \ p\hbox{-practical}\}.$$ Computational data seem to suggest an estimate for the order of magnitude of $F_p(X)$. For example, when $p = 2$, we can use Sage to compute a table of ratios of $F_2(X)/\frac{X}{\log X}.$ 

\begin{table}[h!]

\begin{center}
    \begin{tabular}{ | l | l | c |}
    \hline
    $X$ & $F_2(X)$ & $F_2(X)/(X/\log X)$ \\ \hline
    $10^2$ & 34 & 1.565758\\
    $10^3$ & 243 & 1.678585\\
    $10^4$ & 1790 & 1.648651\\
    $10^5$ & 14703 & 1.692745\\
    $10^6$ & 120276 & 1.661674\\
    $10^7$ & 1030279 & 1.660614\\
    \hline
    \end{tabular}
\end{center}
\caption{Ratios for $2$-practicals}\label{table:f2}
\end{table}

The table looks similar for other small values of $p$. For example, when $p = 3, 5$ we have:

\begin{table}[h!]
\begin{center}
    \begin{tabular}{ | l | l | c |}
    \hline
    $X$ & $F_3(X)$ & $F_3(X)/(X/\log X)$ \\ \hline
    $10^2$ & 41 &1.888120\\
    $10^3$ & 258 & 1.782201\\
    $10^4$ & 1881 & 1.732465\\
    $10^5$ & 15069 & 1.734883\\
    $10^6$ & 127350 & 1.759405\\
    $10^7$ & 1080749 & 1.741962\\
    \hline
    \end{tabular}
\end{center}
\caption{Ratios for $3$-practicals}\label{table:f3}
\end{table}

\begin{table}[h!]
\begin{center}
    \begin{tabular}{ | l | l | c |}
    \hline
    $X$ & $F_5(X)$ & $F_5(X)/(X/\log X)$ \\ \hline
    $10^2$ & 46 & 2.118378\\
    $10^3$ & 286 & 1.975618\\
    $10^4$ & 2179 & 2.006933\\
    $10^5$ & 16847 & 1.939583\\
    $10^6$ & 141446 & 1.954149\\
    $10^7$ & 1223577 & 1.972173\\
    \hline
    \end{tabular}
\end{center}
\caption{Ratios for $5$-practicals}\label{table:f5}
\end{table}

The fact that the sequences of ratios appear to vary slowly suggests the following conjecture:

\begin{conjecture} For each prime $p$, $\lim_{X \rightarrow \infty} F_p(X)/\frac{X}{\log X}$ exists. \end{conjecture}

The strongest bound that we have been able to prove in this vein is as follows:

\begin{theorem}\label{GRH20} Let $p$ be a prime number. Assuming that the Generalized Riemann Hypothesis holds, we have $F_p(X) = O\left(X\sqrt{\frac{\log \log X}{\log X}}\right).$ \end{theorem}

Here we use a version of the Generalized Riemann Hypothesis for Kummerian fields. The dependence on the GRH arises from a lemma of Li and Pomerance that we will use in Section 2. 

For ease of reference, we compile a list of the common notation that will be used throughout this paper. Let $n$ always represent a positive integer. Let $p$ and $q$, as well as any subscripted variations, be primes. Let $P(n)$ denote the largest prime factor of $n$, with $P(1) = 1$. We say that an integer $n$ is $B$-smooth if $P(n) \leq B$. We will use $P^-(n)$ to denote the smallest prime factor of $n$, with $P^-(1) = + \infty$. 

We will use several common arithmetic functions in this body of work. Let $\tau(n)$ denote the number of positive divisors of $n$. We use $\Omega(n)$ to denote the number of prime factors of $n$ counting multiplicity. Lastly, let $\lambda(n)$ denote the Carmichael $\lambda$-function, which represents the exponent of the multiplicative group of integers modulo $n$. 

%%Background and preliminary results
\section{Preliminary lemmas}

In this section, we provide some preliminary lemmas that will be used in the proof of Theorem \ref{GRH20}. We begin by discussing multiplicative orders and their connection to the $p$-practical numbers. Let $\ell_a(n)$ denote the multiplicative order of $a$ modulo $n$ for integers $a$ with $(a, n) = 1$. If $(a, n) > 1$, let $n_{(a)}$ represent the largest divisor of $n$ that is coprime to $a$, and let $\ell_a^*(n) = \ell_a(n_{(a)}).$ In particular, if $(a, n) = 1$, then $\ell_a^*(n) = \ell_a(n).$ In \cite{thompson3}, we gave an alternative characterization of the $p$-practical numbers in terms of the function $\ell_p^*(n)$, which we state here as a lemma:

\begin{lemma}\label{pcondsum} An integer $n$ is \textit{$p$-practical} if and only if every $m$ with $1 \leq m \leq n$ can be written as $m = \sum_{d \mid n} \ell_p^*(d) n_d$, where $n_d$ is an integer with $0 \leq n_d \leq \frac{\varphi(d)}{\ell_p^*(d)}.$\end{lemma} 

Throughout the remainder of this section, let $a > 1$ be an integer and let $A_q$ denote the set of primes $p \equiv 1 \pmod{q}$ with $a^{\frac{p-1}{q}} \equiv 1 \pmod{p}$. We will make use of several lemmas from \cite{lp}, which we state here for the sake of completeness.

\begin{lemma}[Li, Pomerance]\label{lp1} Let $\psi(X)$ be an arbitrary function for which $\psi(X) = o(X)$ and $\psi(X) \geq \log \log X.$ The number of integers $n \leq X$ divisible by a prime $p > \psi(X)$ with $\ell^*_a(p) < \frac{p^{1/2}}{\log p}$ is $O(\frac{X}{\log \psi(X)})$. \end{lemma}

\begin{lemma}[Li, Pomerance]\label{lp2} The number of integers $n \leq X$ divisible by a prime $p \equiv 1 \pmod{q}$ with $$\frac{q^2}{4 \log^2 q} < p \leq q^2 \log^4 q$$ is $O(\frac{X \log \log q}{q \log q}).$ \end{lemma}

\begin{lemma}[Li, Pomerance]\label{lp3}(GRH) Suppose that $q$ is an odd prime and that $a$ is not a $q^{th}$ power. The number of integers $n \leq X$ divisible by a prime $p \in A_q$ with $p \geq q^2 \log^4 q$ is $O\left(\frac{X}{q \log q} + \frac{X \log \log X}{q^2}\right).$ \end{lemma}

Next, we present a version of Proposition 1 from Li and Pomerance's paper \cite{lp}, which will play an important role in obtaining the bound stated in Theorem \ref{GRH20}. As in \cite{lp}, our lemma will make use of Lemma \ref{lp3}; thus, it will depend on the validity of the Generalized Riemann Hypothesis.

\begin{lemma}\label{prop1}(GRH) Let $a$ be a positive integer. Let $\psi(X)$ be defined as in Lemma \ref{lp1}. The number of integers $n \leq X$ with $P(\frac{\lambda(n)}{\ell^*_a(n)}) \geq \psi(X)$ is $O(\frac{X \log \log \psi(X)}{\log \psi(X)}).$ \end{lemma} 

\begin{proof}

Suppose that $n \leq X$ and $q = P(\lambda(n)/\ell^*_a(n)) \geq \psi(X).$ We may assume that $X$ is large, so $a$ is not a $q^{th}$ power and $\psi(X) > a$. Moreover, as we will now show, it must be the case that either $q^2 \mid n$ or $p \mid n$ for some $p \in A_q$. Observe that $$q \mid \frac{\lambda(n)}{\ell^*_a(n)} \mid \frac{\mathrm{lcm}_{p^{e} \mid \mid n} \left[\lambda(p^{e})\right]}{\mathrm{lcm}_{p^{e} \mid \mid n} \left[\ell^*_a(p^e)\right]} \mid \mathrm{lcm}_{p^{e} \mid \mid n} \left[\frac{\lambda(p^e)}{\ell^*_a(p^e)}\right].$$ In particular, $q$ must divide $\frac{\lambda(p^e)}{\ell^*_a(p^e)}$ for some prime $p$. If $q = p,$ then $q \mid \lambda(p^{e})$ implies that $e \geq 2,$ so $q^2 \mid n.$ If $q \neq p$, then $q \mid \frac{\lambda(p)}{\ell^*_a(p)}$, so $p > q > \psi(X) > a$. Thus, $\ell^*_a(p) = \ell_a(p) \mid \frac{p-1}{q}$, so $p \mid a^{\frac{p-1}{q}} - 1,$ which implies that $p \in A_q.$ 

To handle the case where $q^2 \mid n$, we observe that \begin{align*}\#\{n \leq X: q^2 \mid n \ \hbox{for some prime} \ q \geq \psi(X)\} & \leq \sum_{\substack{q \geq \psi(X) \\ q \ \mathrm{prime}}} \left\lfloor \frac{X}{q^2} \right\rfloor \\ & \leq X \sum_{\substack{q \geq \psi(X) \\ q \ \mathrm{prime}}} \frac{1}{q^2} \ll \frac{X}{\psi(X)}.\end{align*} Thus, we may assume that $n$ is divisible by a prime $p \in A_q$ with $p > a$.

By Lemma \ref{lp1}, we may assume that $\ell^*_a(p) \geq p^{1/2}/\log p.$ However, since $p \in A_q$ implies that $a^{\frac{p-1}{q}} \equiv 1 \pmod{p}$, then $\ell_a(p) \leq \frac{p-1}{q}$, so $p > \frac{q^2}{(4 \log^2 q)}.$ Thus, we can use Lemmas \ref{lp2} and \ref{lp3} to deal with the remaining values of $n \leq X$. In particular, we have \begin{align}\notag & \#\{n \leq X: p \mid n \ \hbox{for some} \ p \in A_q \ \hbox{with} \ p > q^2/(4 \log^2 q)\} \\ \notag & \leq \#\{n \leq X: p \mid n \ \hbox{for some} \ p \equiv 1 \pmod{q} \ \hbox{with} \ p \in (\frac{q^2}{4 \log^2 q}, q^2 \log^4 q]\} \\ \notag & + \#\{n \leq X: p \mid n \ \hbox{for some} \ p \in A_q \ \hbox{with} \ p \geq q^2 \log^4 q\} \\ & \label{lp23} \ll \frac{X \log \log q}{q \log q} + \frac{X}{q \log q} + \frac{X \log \log X}{q^2}, \end{align} where the final inequality follows from Lemmas \ref{lp2} and \ref{lp3}. Since our hypotheses specify that $q \geq \psi(X)$, then the bound given in \eqref{lp23} implies \begin{align*} & \#\{n \leq X: q \geq \psi(X) \ \hbox{and} \ p \mid n \ \hbox{for some} \ p \in A_q\} \\ & \ll X \sum_{q \geq \psi(X)} \left(\frac{\log \log q}{q \log q} + \frac{\log \log X}{q^2}\right) \\ & \ll \frac{X \log \log \psi(X)}{\log \psi(X)}.\end{align*} \end{proof}

%%%Key lemma
\section{Key lemma}

The key to proving Theorem \ref{GRH20} rests in showing that $\ell_p^*(n)$ is usually not too small. We make this statement precise with the following lemma: 

\begin{lemma}\label{cor2}(GRH) Let $\theta$ be a constant satisfying $\frac{1}{10} \leq \theta \leq \frac{9}{10}$. Let $Y = e^{110(\log X)^\theta (\log \log X)^2}.$ For all $a > 1$ and $X$ sufficiently large, uniformly in $\theta$, we have \begin{equation}\label{GRHppractical}\#\{n \leq X: \ell^*_a(n) \leq \frac{X}{Ye^{(\log X)^\theta}}\} \ll \frac{X}{(\log X)^{\theta} \log \log X}.\end{equation} \end{lemma}

Before we prove Lemma \ref{cor2}, we will introduce three additional results, the first of which is due to Friedlander, Pomerance and Shparlinski \cite{fps} and the last of which is due to Luca and Pollack \cite{lucapollack}. 

\begin{lemma}\label{originalfrposh} For sufficiently large numbers $X$ and for $\Delta \geq (\log \log X)^3$, the number of positive integers $n \leq X$ with $$\lambda(n) \leq n \exp(-\Delta)$$ is at most $X \exp (-0.69 (\Delta \log \Delta)^{1/3}).$\end{lemma}

\begin{corollary}\label{frposh} Let $\theta$ be as in Lemma \ref{cor2}. For sufficiently large $X$, the number of positive integers $n \leq X$ with $$\lambda(n) \leq \frac{X}{e^{(\log X)^\theta}}$$ is at most $X/e^{(\log X)^{\theta/3}}.$ \end{corollary}

\begin{proof} Trivially, there are at most $X/\exp((\log X)^{\theta/2})$ values of $n \leq X/\exp((\log X)^{\theta/2})$ with $\lambda(n) \leq X/\exp((\log X)^\theta).$ On the other hand, if $X/\exp((\log X)^{\theta/2}) < n \leq X$, then $X \leq n \exp((\log X)^{\theta/2}).$ Thus, for large $X$, we have \begin{align*}\#\left\{\frac{X}{e^{(\log X)^{\theta/2}}} < n \leq X : \lambda(n) \leq \frac{X}{e^{(\log X)^\theta}}\right\} & \leq \#\left\{ n \leq X : \lambda(n) \leq \frac{n e^{(\log X)^{\theta/2}}}{e^{(\log X)^\theta}}\right\} \\ & < \#\left\{n \leq X : \lambda(n) \leq \frac{n}{e^{\frac{1}{2}(\log X)^\theta}}\right\}.\end{align*} Applying Lemma \ref{originalfrposh} with $\Delta = \frac{1}{2}(\log X)^\theta$, we see that this is at most $X/\exp(2(\log X)^{\theta/3}).$ Therefore, $$\#\left\{n \leq X: \lambda(n) \leq \frac{X}{e^{(\log X)^\theta}} \right\} \leq \frac{X}{e^{(\log X)^{\theta/2}}} + \frac{X}{e^{2(\log X)^{\theta/3}}} \leq \frac{X}{e^{(\log X)^{\theta/3}}}.$$

\end{proof}

\begin{lemma}\label{lucapollack} But for $O(\frac{X}{(\log X)^3})$ choices of $n \leq X$, we have $$\Omega(\varphi(n)) < 110(\log \log X)^2.$$ \end{lemma}

We will use these results in the proof of Lemma \ref{cor2}, which we present below.

\begin{proof} Let $\theta$ be such that $\frac{1}{10} \leq \theta \leq \frac{9}{10}$, let $B = e^{(\log X)^{\theta}}$ and let $u(n)$ denote the $B$-smooth part of $\lambda(n)$. Let $Y$ be defined as in the statement of Lemma \ref{cor2}. If $\lambda(n)$ has a large $B$-smooth part, say $u(n) > Y$, then so does $\varphi(n)$, since $u(n)$ must divide $\varphi(n)$ as well. First, we will estimate the number of $n \leq X$ for which $u(n) > Y.$ Let $\Omega(u(n)) = k$. By definition, all prime factors of $u(n)$ are at most $e^{(\log X)^{\theta}}$. Thus, we have $$Y < u(n) \leq (e^{(\log X)^{\theta}})^k.$$ Solving for $k$, we obtain $k \geq 110(\log \log X)^2$. However, Lemma \ref{lucapollack} implies that $k < 110 (\log \log X)^2$ except for $O(\frac{X}{(\log X)^3})$ values of $n \leq X.$ Hence, we can conclude that there are at most $O(\frac{X}{(\log X)^3})$ values of $n$ for which the $B$-smooth part of $\lambda(n)$ is larger than $Y.$ Thus, using Lemma \ref{frposh}, we have \begin{align*} \#\{n \leq X: \frac{\lambda(n)}{u(n)} \leq \frac{X}{Ye^{(\log X)^\theta}}\} & \leq \#\{n \leq X: \lambda(n) \leq \frac{X}{e^{(\log X)^\theta}}\} + \#\{n\leq X: u(n) > Y\} \\ & \ll \frac{X}{e^{(\log X)^{\theta/3}}} + \frac{X}{(\log X)^{3}}.\end{align*} However, if we take $\psi(X) = Y \exp\{(\log X)^{\theta}\}$ then we can use Lemma \ref{prop1} to show that, for all but $O(\frac{X}{(\log X)^{\theta} \log \log X})$ choices of $n \leq X$, we have $\frac{\lambda(n)}{u(n)} \mid \ell^*_a(n).$ Therefore, we have $$\ell^*_a(n) \geq \frac{\lambda(n)}{u(n)} > \frac{X}{Ye^{(\log X)^\theta}},$$ except for at most $O(\frac{X}{(\log X)^{\theta} \log \log X})$ values of $n \leq X$. \end{proof}

%%%Density considerations for $p$-practical numbers
\section{Proof of Theorem \ref{GRH20}}

In this section, we present the proof of our main theorem. We begin by discussing the remaining lemmas that we will need in order to complete the argument. Let $n$ be a positive integer, with $d_1 < d_2 < \cdots < d_{\tau(n)}$ its increasing sequence of divisors. Let $Z \geq 2.$ We say that $n$ is $Z$-dense if $\max_{1 \leq i \leq \tau(n)} \frac{d_{i+1}}{d_i} \leq Z$ holds. The following lemma, due to Saias (cf. \cite[Theorem 1]{saias}), describes the count of integers with $Z$-dense divisors.

\begin{lemma}[Saias]\label{saias1} For $X \geq Z \geq 2$, we have \begin{align}\label{saiasppractical}\# \{n \leq X: n \ \hbox{is} \ Z\hbox{-dense}\} \ll \frac{X \log Z}{\log X}.\end{align} \end{lemma}

The next lemma is due essentially to Friedlander, Pomerance and Shparlinski (cf. \cite[Lemma 2]{fps}).

\begin{lemma}\label{l2} Let $n$ and $d$ be positive integers with $d \mid n$. Then, for any rational prime $p$, we have $\frac{d}{\ell^*_p(d)} \leq \frac{n}{\ell^*_p(n)}.$ \end{lemma}

\begin{proof} The result is proven in \cite{fps} when $(p, n) = 1$. In the case where $(p, n) > 1$, let $n_{(p)}$ and $d_{(p)}$ represent the largest divisors of $n$ and $d$ that are coprime to $p$, respectively. Then $$\frac{d}{d_{(p)}} \leq \frac{n}{n_{(p)}},$$ since the highest power of $p$ dividing $d$ is at most the highest power of $p$ dividing $n$. After a rearrangement, we have $$\frac{d}{n} \leq \frac{d_{(p)}}{n_{(p)}} \leq \frac{\ell^*_p(d)}{\ell^*_p(n)},$$ where the final inequality follows from the coprime case.\end{proof} 

We will also use the following elementary lemma:

\begin{lemma}\label{kappa} Let $X \geq 2$ and let $\kappa \geq 1.$ Then, we have $$\#\{n \leq X : \tau(n) \geq \kappa\} \ll \frac{1}{\kappa} X \log X.$$ \end{lemma}

\begin{proof} We observe that $$\sum_{n \leq X} \tau(n) = \sum_{n \leq X} \sum_{d \mid n} 1 \leq X \sum_{d \leq X} \frac{1}{d} \ll X \log X.$$ The number of terms in the sum on the left-hand side of the equation that are $\geq \kappa$ is $\ll \frac{1}{\kappa} X \log X.$ \end{proof}

Now we have all of the tools needed to prove Theorem \ref{GRH20}. Below, we present its proof. 

\begin{proof}

Let $n$ be a positive integer with divisors $d_1 < d_2 < \cdots < d_{\tau(n)}.$ Let $p$ be a rational prime with $p \nmid n$. Let $\theta$ and $Y$ be as in Lemma \ref{cor2}. In \eqref{saiasppractical}, set $Z = Y^2$. Assume that $n$ is not in the set of size $O(X \log Y^2/ \log X)$ of integers with $Y^2$-dense divisors. Then there exists an index $j$ with \begin{equation}\label{zdense} \frac{d_{j+1}}{d_j} > Y^2.\end{equation} Moreover, we can use Lemma \ref{kappa} to show that \begin{equation}\label{taulogx} \# \{n \leq X: \tau(n) > Y/e^{(\log X)^\theta}\} \ll \frac{X e^{(\log X)^\theta} \log X }{Y}.\end{equation} As a result, we will assume hereafter that $\tau(n) \leq Y/e^{(\log X)^\theta}.$ Examining the ratios $\frac{d_{k+1}}{d_k}$, we remark that it is always the case that $d_1 = 1$ and $d_2 = P^-(n);$ hence, we have $$\#\{n \leq X : \frac{d_2}{d_1} > Y^2\} = \sum_{\substack{n \leq X \\ P^-(n) > Y^2}} 1 \ll X \prod_{q \leq Y^2} \left(1 - \frac{1}{q}\right),$$ where the final inequality follows from applying Brun's Sieve (cf. \cite[Theorem 2.2]{hr}). By Mertens' Theorem (cf. \cite[Theorem 3.15]{pollack}), we have \begin{equation}\label{brunexcept} X \prod_{q \leq Y^2} \left(1 - \frac{1}{q}\right) \ll \frac{X}{\log Y}.\end{equation} 

Now, suppose that $k > 1$. On one hand, for all $k > 1$, we have \begin{equation}\label{lpeq1} 1 + \sum_{l \leq k} \ell^*_p(d_l) \frac{\varphi(d_l)}{\ell^*_p(d_l)} = 1 + \sum_{l \leq k} \varphi(d_l) \leq k d_k \leq Ye^{-(\log X)^\theta} d_k.\end{equation} On the other hand, Lemma \ref{cor2} implies that $\ell^*_p(n) > \frac{X}{Ye^{(\log X)^\theta}}$ but for \begin{align}\label{bigohppracexcept} O\left(\frac{X}{(\log X)^{\theta} \log \log X}\right)\end{align} integers $n \leq X$. For such numbers $n$, for all $i \geq 1$, we have \begin{align}\label{lpeq2} \ell^*_p(d_{j+i}) \geq \frac{\ell^*_p(n) d_{j+i}}{n} > \frac{d_{j+i}}{Y e^{(\log X)^\theta}} > \frac{d_j Y^2}{Y e^{(\log X)^\theta}} = Ye^{-(\log X)^\theta} d_j \end{align} where the inequalities follow, respectively, from Lemma \ref{l2}, Lemma \ref{cor2} and the assumption that there exists an index $j$ for which \eqref{zdense} holds. As a result, we can combine the inequality from \eqref{lpeq1} applied with $k = j$ with\eqref{lpeq2} to show that $$1 + \sum_{l \leq j} \ell^*_p(d_l) \frac{\varphi(d_l)}{\ell^*_p(d_l)} < \ell^*_p(d_{j+i})$$ holds for all $i \geq 1$. Thus, $x^n-1$ has no divisor of degree $1 + \sum_{l \leq j} \varphi(d_l)$ in $\F_p[x]$, so $n$ is not $p$-practical. Therefore, by \eqref{saiasppractical}, \eqref{taulogx}, \eqref{brunexcept} and \eqref{bigohppracexcept}, we have \begin{align}\label{finalfpineq} F_p(X) \ll \frac{X \log Y}{\log X} + \frac{X e^{(\log X)^\theta} \log X}{Y} + \frac{X}{\log Y} + \frac{X}{(\log X)^\theta \log \log X}.\end{align} Now, the only significant terms in \eqref{finalfpineq} are $\frac{X}{(\log X)^\theta \log \log X}$ and $\frac{X \log Y}{\log X}$. Equating these expressions and using the fact that $Y = e^{110(\log X)^\theta(\log \log X)^2}$, we obtain $\theta = \frac{1}{2} - \frac{3 \log_3 X}{2 \log_2 X}$ as a good choice for $\theta$. Plugging this value of $\theta$ into the bound $\frac{X}{(\log X)^\theta \log \log X}$ yields a bound of $O\left(X \sqrt{\frac{\log \log X}{\log X}}\right)$ for the size of the set of $p$-practicals up to $X$.\end{proof}

\textit{Acknowledgements.} The work contained in this paper comprises a portion of my Ph.D. thesis \cite{thompsonthesis}. I would like to thank my adviser, Carl Pomerance, for his guidance throughout the process of completing this work. I would also like to thank Paul Pollack for pointing out the relevant results in \cite{lucapollack} and for his help with simplifying a few of my arguments.

\end{document}